\documentclass{amsart}
\usepackage[foot]{amsaddr}
\usepackage{ulem}

\RequirePackage[colorlinks=false, urlcolor=cyan, citecolor=magenta, breaklinks=true]{hyperref}

\usepackage{amsmath}
\usepackage{amsfonts}
\usepackage{amssymb}
\usepackage{amsthm}
\usepackage{graphicx}
\usepackage{xcolor}
\usepackage{cleveref}
\usepackage{enumitem}

\usepackage{bm}
\usepackage{mathtools}

\usepackage{leftidx}
\usepackage{isomath}
\usepackage[usenames,dvipsnames]{pstricks}
\usepackage{epsfig}
\usepackage{pst-grad} 
\usepackage{pst-plot} 
\usepackage{textcomp} 
\usepackage{nicefrac}
\usepackage{wrapfig}
\usepackage{algorithm}
\usepackage{algpseudocode}
\usepackage{framed}
\usepackage{float}

\usepackage[hyphenbreaks]{breakurl}
\usepackage{orcidlink}
\usepackage{datetime2}

\newcommand{\remove}[1] {}
\theoremstyle{plain}
\newtheorem{definition}{Definition}

\newtheorem{lemma}{Lemma}
\newtheorem{corollary}{Corollary}
\newtheorem{theorem}{Theorem}

\newtheorem*{fact*}{Fact}

\newtheorem{remark}{Remark}
\newtheorem*{remark*}{Remark}

\newtheorem*{example*}{Example}
\newtheorem*{theorem*}{Theorem}
\newtheorem*{maintheorem*}{Main Theorem}
\newtheorem*{corollary*}{Corollary}
\newtheorem*{summary*}{Summary of results}
\newtheorem*{importantrmk*}{Important remark}

\newcommand{\K}{(2+\gamma)(\Delta-1)+1}

\newcommand{\D}{\Delta}

\newcommand{\x}{\chi_a}
\newcommand{\EC}{\textsc{EdgeColor}}

\newcommand{\ValE}{\textsc{EdgeValidation}}

\newcommand{\Recolor}{\textsc{Recolor}}
\newcommand{\E}{\mathcal{E}}
\newcommand{\C}{\mathcal{C}}

\newcommand{\F}{\mathcal{F}}

\newcommand{\sco}{\mathrm{sc}}

\newcommand{\T}{\mathcal{T}}
\definecolor{dred}{rgb}{0.6, 0, 0}
\definecolor{dblue}{rgb}{0, 0, 0.6}
\definecolor{blue}{rgb}{0, 0, 0.9}

\newcommand{\nicetrees}{T}
\newcommand{\nicetreesc}{\mathcal{T}}
\DeclareMathOperator{\MSet}{\textbf{MSet}}
\newcommand{\MSetOdd}{\MSet_{-odd}}
\newcommand{\MSetTwo}{\MSet_{2}}
\newcommand{\UG}{\hat{G}}
\newcommand{\LG}{\check{G}}

\definecolor{dred}{rgb}{0.6, 0, 0}
\definecolor{dblue}{rgb}{0, 0, 0.6}
\definecolor{blue}{rgb}{0, 0, 0.9}

\usepackage[textsize=tiny]{todonotes}

\makeatletter
\def\namedlabel#1#2{\begingroup
   \def\@currentlabel{#2}%
   \label{#1}\endgroup
}
\makeatother

\date{\today}

\title[Acyclic chromatic index]{An improvement on the bound for the acyclic chromatic index
}

\author{Lefteris Kirousis \orcidlink{0000-0002-4912-8959}}\thanks{Corresponding author: L. Kirousis.}
\address{National and Kapodistrian University of Athens, Greece}
\email[Lefteris Kirousis]{lkirousis@math.uoa.gr}

\author{John Livieratos \orcidlink{0000-0001-6409-4286}}
\address{National and Kapodistrian University of Athens, Greece}
\email[John Livieratos]{jlivier89@math.uoa.gr}

\author{Alexandros Singh \orcidlink{0000-0002-6263-950X}}
\address{Universit\'{e} Paris 8, France}
\email[Alexandros Singh]{as@up8.edu}

\begin{document}
\maketitle


\begin{abstract}
 The {\em acyclic chromatic index} (or acyclic edge-chromatic number) of a graph
is the least number of colors needed to properly color its edges so that none
of its cycles has only two colors. We show that for a graph of max degree
$\Delta$, the acyclic chromatic index is at most $3.142(\Delta-1)+1$,
improving on the (best to date) bound of Fialho et al. (2020).
Our improvement is made possible by considering unordered
(non-plane) trees, instead of ordered (plane) ones, as witness structures for
the Lov\'{a}sz Local Lemma, a key combinatorial tool often used in related
works. The counting of these witness structures entails methods of Analytic Combinatorics.
\end{abstract}




\section{Introduction}\label{sec:intro}

The  {\em chromatic index} of $G$ is the least number of colors needed to properly color the edges of $G$, i.e., to color them so that no edges coincident on a common vertex get the same color.  It is known that the chromatic index of any graph is either $\Delta$ or $\Delta+1$ (Vizing \cite{vizing1965critical}), where $\Delta$ denotes the max degree of the graph. Nevertheless, observe that to generate a proper edge coloring by successively coloring its edges in a way that at each step, the chosen color does not violate properness and no color is ever changed necessitates a palette of at least $2\Delta -1$ colors, because $2\Delta -2 $ edges might be coincident with any given edge. 

The {\em acyclic chromatic index} of $G$, here denoted by $\chi'_{\alpha}(G)$,   is the least number of colors needed to properly color the edges of $G$ so that no cycle of even length is bichromatic (in any properly edge-colored graph, no cycle of odd length can be bichromatic).  It has been conjectured (J. Fiam\v{c}ik \cite{fiam} and Alon et al. \cite{alon2001acyclic}) that the acyclic chromatic index is at most  $\Delta +2$. 
 
 Besides the numerous publications for special cases of graphs, the  literature  on the acyclic chromatic index for general graphs with max degree $\Delta$  includes:  
 \begin{itemize}
 \item 	Alon et al. \cite{alon1991acyclic} proved $\chi'_{\alpha}(G) \leq 16\Delta$ (see also Molloy and Reed \cite[Theorem~2.2]{molloy1998further}).  Ndreca et al. \cite{{ndreca2012improved}}   
 improved this to  $\chi'_{\alpha}(G) \leq  \lceil9.62(\Delta-1)\rceil$.  
 \item Subsequently, Esperet and Parreau \cite{DBLP:journals/ejc/EsperetP13}  proved that $\chi'_{\alpha}(G) \leq 4( \Delta -1)$.   
 \item The latter bound was improved to  $\lceil3.74(\Delta -1)\rceil +1$ by Giotis et al. \cite{giotis2017acyclic}.  Also, an improvement of the $4( \Delta -1)$ bound was announced by  Gutowski et al. \cite{Gutowski2018} (the specific coefficient for $\Delta$ is not given in the abstract of the announcement). 
 \item  Finally, the  best bound until now  was given by Fialho et al. \cite{fialho2020new}, who proved that $\chi'_{\alpha}(G) \leq 3.569(\Delta -1)$.   \end{itemize} Here we show that $$\chi'_{\alpha}(G) \leq 3.142(\Delta-1) +1.$$
 The most recent results above are based on the algorithmic proof of the
Lov\'{a}sz Local Lemma by Moser \cite{moser2009constructive} and
Moser and Tardos \cite{moser2010constructive}, which uses an approach that has
been known as the {\em entropy compression method}. The use of this technique
to generate an acyclic coloring of a given graph goes, roughly and
qualitatively, as follows. We start by successively coloring all edges, each
time choosing uniformly at random (u.a.r.) among the set of colors that do not
violate properness. We then repeatedly recolor any bichromatic cycle that pops
up, with the choice of each bichromatic cycle being such that we generate a
{\em witness forest} whose nodes correspond to the chosen cycles. It turns out that it is
highly improbable to continue locating bichromatic cycles for long, i.e. it is
highly improbable to construct large witness forests, lest we get a
disagreement between the entropy of the  cycles that are structured as trees and the probability
of their being bichromatic.

 In previous results, the witness forests considered were ordered (plane).
Here we consider unordered (non-plane) trees, so we further restrict the
entropy. This necessitates estimating the number of degree-restricted, rooted,
unlabeled, unordered trees, which we do using a consequence of the smooth implicit-function schema presented by Flajolet and Sedgewick in \cite[Theorem VII.3, p. 268]{Flajolet:2009:AC:1506267}. 

 To prove the eventual halting of the algorithm, we use the technique of Giotis
et al.~\cite{giotis2017acyclic}, a special case of the {\em forward-looking
analysis} introduced later by  Harvey and Vondr\'{a}k \cite{harvey2015algorithmic}. Also, the probabilistic computations in   Giotis et al.  are   used to compute the probability of a cycle being bichromatic, as   Fialho et al. \cite{fialho2020new} later do.

We choose the colors so that a stronger than properness property is retained, namely that no bichromatic 4-cycles ever appear (for the reason, see Remark \ref{rmk:4-ac}).

 \subsection{Notation and terminology}\label{ssec:not}  In the sequel, we give some general notions and introduce the notation and terminology we use.

Throughout this paper, $G$ is a simple graph with $l$ vertices and $m$ edges, sometimes referred to as the {\em underlying graph}. All parameters pertaining to this graph are considered {\it constants}.  We use the letter $n$ mainly to denote  the number of  nodes of the witness structure an algorithm generates (a notion defined below). It is with respect  to  $n$ that we make asymptotic considerations.

The \emph{maximum degree} of $G$ is denoted by $\D$ and we assume, to avoid trivialities, that it is greater than $1$.

A $k$-cycle is a succession of $k\geq 3$ distinct vertices $u_1, \ldots, u_k$ and edges $e_1, \ldots,e_k$ so that for $i=1, \ldots, k-1$, $e_i$ connects  $u_i$ and $u_{i+1}$ and $e_k$ connects $u_{k}$ and $u_1$. Among the two possible consecutive traversals (cyclic orderings) of the edges  of  a cycle, we arbitrarily select one and call it {\em positive}. In the sequel, all cycles are considered in their positive traversal. The number of nodes (or edges)  in a cycle $C$  is denoted by $|C|$.

We assume, without loss of generality, that the sets of all
edges of the underlying graph $G$ are ordered by
some {\em predetermined order}. It is with respect to this order that 
we use terms such as ``first''  or ``successively"  for the elements of a given set of edges.  Analogously, for the cycles of the graph.
 
\begin{definition}\label{def:parity}
	Two edges  of a cycle of even length that  are separated by an odd number of edges  are said to have the same {\em parity}. 
\end{definition}

\begin{definition}
Given an edge $e$ and a cycle $C$ with edges in its positive traversal $(e=e_1, e_2, \ldots, e_s)$, the succession  of edges $(e=e_1, e_2, \ldots, e_{s-2})$ (i.e all but the last two) is caled the scope of the pair (e,C), and is denoted by $\sco(e,C)$. 	
\end{definition}
Obviously,  $|\sco(e,C)|$, the cardinality of the scope, is equal to   $|C|-2$.	Also always $e \in \sco(e,C)$.

An edge coloring, or simply a coloring, of $G$ is an assignment of colors to its edges chosen from a given ``palette" (set)  of colors.  
A coloring is  \emph{proper} if no adjacent edges have the same color.  All colorings considered will be proper, so in the sequel we omit this specification.

A cycle of a colored graph is called {\em bichromatic} if its edges  are colored
 by only two colors. In other words, if each edge in $C$ but two of opposite parity is homochromatic with the one among the excluded two with the same parity.  A  coloring is called $k$-acyclic if there are no bichromatic $k$-cycles and acyclic if there are no bichromatic cycles at all. Note that for a cycle to be bichromatic, its length must be even. The \emph{acyclic chromatic index} of $G$, denoted by $\x'(G)$, is the least number of colors needed for an acyclic coloring of $G$. 

A {\em rooted} tree is a tree equipped with a marked vertex, its {\em root}.
Following the usual terminology of enumerative combinatorics (as presented in,
e.g. Flajolet and Sedgewick \cite{Flajolet:2009:AC:1506267}) we say that a class of trees is {\em
unlabeled} if its elements are taken up to permutation of their vertex labels
and {\em unordered} (or {\em non-plane}) if all permutations of subtrees induced by
any given vertex are considered as generating the same tree. If instead we
assume that  different permutations as generating distinct trees we obtain the notion
of {\em ordered} (or {\em plane}) trees. The justification for the last
term is  that fixing a linear order of subtrees induced by each vertex
determines an embedding of said tree on the plane.
A rooted, unlabeled, unordered (non-plane) forest is an unordered
collection of rooted, unlabeled, unordered trees. Analogously for ordered (plane) forests. 
All trees and forests in this work are rooted and unlabeled, so we sometimes delete these specifications. However, they could be either ordered or unordered. All generating functions are ordinary.

\section{Algorithm \EC}\label{sec:color}
We first  give  a cornerstone result proven by Esperet and Parreau in~\cite{DBLP:journals/ejc/EsperetP13}:
\begin{lemma}[Esperet and Parreau~\cite{DBLP:journals/ejc/EsperetP13}]\label{lem:sufficientcolors} At any step of any successive coloring of the edges of a graph, in any order, there are at most 
$2 (\Delta-1)$ colors that should be avoided in order to produce a 4-\nolinebreak acyclic coloring. 
\end{lemma}
\begin{proof}[Proof Sketch]
Notice that for each edge $e$, one has to avoid the colors of all edges  adjacent to $e$, and moreover  for each pair of homochromatic (of the same color) edges $e_1, e_2$  adjacent to $e$ at different endpoints (which contribute one to the count of colors to be avoided), one has also to avoid the color of the at most one edge $e_3$ that together with $e,e_1, e_2$ define a cycle of length 4. Thus, easily, the total count of colors to be avoided does not exceed the number of adjacent edges of $e$, which is at most $2(\Delta-2)$. \end{proof}
\begin{remark}\label{rmk:4-ac}
The fact that by Lemma \ref{lem:sufficientcolors} 
4-acyclicity can be attained
by avoiding  no more colors than those needed to attain properness, allows us to introduce  an algorithm that selects colors among those that also retain 4-acyclicity (besides properness). 
\end{remark}
In the sequel, we assume that our  palette has  at least $\K$ colors for some $\gamma \geq
1$. We will compute the least such $\gamma$ for which there is an acyclic
edge-coloring. It will turn out that $\gamma  \geq 1.142$


We now give the following algorithm, which {\it if and when} it halts, obviously produces an acyclic coloring.

\begin{algorithm}[H]
\caption{\EC}\label{alg:ec}
\vspace{0.1cm}
\begin{algorithmic}[1]
\Statex \underline{Input:} the graph $G$ for which an acyclic coloring is to be constructed.
\State Successively assign to each edge of $G$  a color selected u.a.r. from the set of  \Statex \hspace{3.6em} colors whose assignment  does not violate 4-acyclicity  \label{ec:color}
\While{there is an edge in a  bichromatic cycle} \label{ec:while}
\State let $e$ be the first such edge and $C$ the first such cycle \Comment{$C$ is necessarily \Statex \hspace{3.6em}   of even length $\geq 6$}
\State \Recolor($e,C$)\label{ec:recolor}
\EndWhile\label{ec:endwhile}
\State \textbf{return} the current coloring.
\end{algorithmic}
\begin{algorithmic}[1]
\vspace{0.1cm}
\Statex \underline{\Recolor($e, C$)} 
\State Successively assign to  each edge of $\sco(e,C)$  a color selected u.a.r. from the set   \Statex \hspace{3.6em} of colors whose assignment  does not violate 4-acyclicity \label{ecrecolor:color}
\While{$\sco(e,C)$ shares an edge with a bichromatic cycle} 
\State let $e'$ be the first such edge and $C'$ the first such cycle \Comment{$C'$ could be $C$}
\label{ecrecolor:while}
\State \Recolor($e', C'$)\label{ecrecolor:recolor}
\EndWhile\label{ecrecolor:endwhile}
\end{algorithmic}
\end{algorithm}

The calls of   \Recolor\ of an   execution of \EC\  generate  an unordered   forest  as follows:
the calls made at line \ref{ec:color} of \EC\ generate the roots of the trees of the forest;    whereas the calls of \Recolor\ made at line \ref{ecrecolor:color} of a preceding call of \Recolor\ generate the children of the node generated by  this preceding call. We mark each of the nodes of this forest with the pair $(e,C)$ that is the input of the corresponding call of \Recolor. The edge $e$ is called the edge-mark of the node, whereas $C$ is called its cycle-mark.  We then add to this forest a minimum number of leaves,  so that all nodes corresponding to calls of \Recolor$(e,C)$ become internal with $|C|-2$ children. The leaves are left unmarked.   Thus,  all the internal nodes of the forest generated by an execution of \EC\ have even degree $\geq 4$, since only cycles of even length $\geq 6$ that become bichromatic are recolored by \EC\ (the bichromaticity of cycles of length 4 is avoided when choosing colors).

\begin{definition}\label{def:witness}
We call {\em witness forest} of an execution of \EC\ the marked and  unordered forest generated by the execution.	Obviously, the witness forest of an execution  is uniquely defined, however it is conceivable that different executions generate the same witness forest. The witness forest is a random variable. 
\end{definition}

\begin{remark}\label{rmk:marks}
The marking of the nodes of the witness forest does not render the forest as a labeled one. The number of possible witness forests will be computed without taking into account their marks. 
\end{remark}

 Before delving into the probabilistic analysis of \EC, let us show the following lemma. We call it the progress lemma, because it intuitively states that at every time a call of \Recolor\ terminates, some progress is  made, which is then preserved in subsequent phases.
\begin{lemma}[Progress Lemma] \label{lem:progr}
Consider a call of \Recolor\ applied to  $(e,C)$ and  let $\mathcal{E}_0$ 
be the set of edges that at the beginning of the call do not belong to a bichromatic cycle.  Then, if and when the {\bf while}-loop of  \Recolor($e,C$) terminates, no edge in   $\mathcal{E}_0\cup \sco(e,C)$ belongs to a bichromatic cycle. \end{lemma}
\begin{proof}
Suppose that the {\bf while}-loop of \Recolor($e,C$) terminates and there is some  edge   $e'\in\E_0\cup \sco(e,C)$ that belongs to  a bichromatic cycle.

If $e'\in \sco(e,C)$, then   \Recolor($e,C$) could not have terminated. So assume, towards a contradiction,  that $e'\in\E_0$.

So, it must be the case that at some point during the {\bf while}-loop of \Recolor$(e,C)$, a cycle $C'$ that contains $e'$ turned bichromatic and remained so until the end of \Recolor($e,C$). Assume that $C'$ became bichromatic because of  some call of \Recolor\ within \Recolor$(e,C)$. Then the {\bf while}-loop of this call  could not terminate, because $C'$  stays bichromatic. Therefore \Recolor$(e,C)$ does not terminate either, a contradiction. 
\end{proof}
By the Progress Lemma \ref{lem:progr}, we easily get that:

\begin{corollary}\label{cor:witnessProp}
A witness forest satisfies:

\begin{enumerate}
\item\label{ite:fir} pairwise,  the scopes of the cycle-marks of the roots   do not share an edge, so the edge-marks of the roots are pairwise distinct.
\item \label{ite:sec} pairwise the  cycle-marks of the children of every node do not share an edge, so the corresponding edge-marks are pairwise distinct.
\item \label{ite:thi} if $(e,C)$ is the mark of a node $u$, the edge-mark of any child of $u$ belongs to $\sco(e,C)$. 
\end{enumerate}
\end{corollary}

The Progress Lemma \ref{lem:progr}, however,  does not exclude  the possibility of  a call of \Recolor\  lasting infinitely long  by recycling  sub-calls of \Recolor;   therefore this lemma  does not exclude the possibility of an execution of \EC\ never halting. 

At this point, let us remind the reader that a function of $n$ is said to be inverse exponential  (or decay exponentially) in $n$, if it is $O(c^n)$ for some positive constant $c< 1$. Any constant {\em positive} power (even with exponent $<1$) of an inverse exponential in $n$ function is inverse exponential in $n$  and  multiplying an inverse exponential with a polynomial gives an inverse exponential.

Let now $P_n$ be the probability that \EC\ generates a witness forest with  $n$ nodes. In Subsection \ref{ssec:vale},
we will prove  the following:

\begin{lemma} \label{factt1} If at least $\K$ are available for some $\gamma \geq 1.142$, then $P_n$    is inverse exponential in $n$.
\end{lemma}

Because the outdegrees of internal nodes of the witness forest are constants (the lengths of cycles are considered constants), the probability that \EC\ generates a witness forest with  $n$ nodes (leaves included) is inverse exponential iff the probability that \EC\ generates a witness forest with at $n$ internal nodes is inverse exponential. The latter is true iff the probability that \EC\ makes  $n$ calls of \Recolor\ is inverse exponential, which holds iff the probability that \EC\ makes at least  $n$ calls of \Recolor\ is inverse exponential.

Therefore, from Lemma \ref{factt1}, by the probabilistic method (to prove existence it suffices to prove that the probability of existence is positive),   we   immediately get the theorem below,  
our main result.
\begin{maintheorem*}\namedlabel{maintheorem}{Main Theorem}
$\lceil 3.142 (\Delta -1) \rceil +1$ colors suffice to properly and acyclically color a graph.	\end{maintheorem*}

\section{Validation Algorithm}\label{ssec:vale}
Towards finding an upper    bound of the probability that   \EC\ generates a witness forest with $n$ nodes,
we consider  the   algorithm \ValE\ below that takes as input an arbitrary  sequence $$\mathcal{S} = (e_1^1, e_1^2, k_1),\ldots, (e_s^1, e_s^2, k_{s}),$$ such that  
for all $i=1, \ldots, s$, there is a cycle $C_i$ of length $2k_i\geq 6$ that contains $e_i^1$ and  $e_i^2$ as successive
edges  in its positive traversal when we start from $e_i^1$. We call such sequences {\em admissible sequences.}

\renewcommand{\labelenumi}{(\alph{enumi})}
\renewcommand{\theenumi}{(\alph{enumi})}
\algnewcommand{\myIf}[1]{\State\algorithmicif\ #1}
\algnewcommand{\myThen}[1]{\State\algorithmicthen\ #1}
\algnewcommand{\myEndIf}{\State \algorithmicend\ \algorithmicif}
\algnewcommand{\myElse}[1]{\State\algorithmicelse\ #1}
\newcommand\algotext[1]{\end{algorithmic}#1\begin{algorithmic}[1]}

\begin{algorithm}[H]
\caption{\ValE$(\mathcal{S})$}\label{alg:vale}
\begin{algorithmic}[1]
\Statex \underline{Input:} An admissible sequence $\mathcal{S} =   (e^1_i, e^2_i, k_i), i=1, \ldots, s$ of  $G$.
\State Successively assign to each edge of $G$  a color selected u.a.r. from the set of  \Statex \hspace{3.6em} colors whose assignment  does not violate 4-acyclicity  \label{val:color}
\State $\text{success} := \text{\tt true}$
\State $i :=1$
\While {$i\leq s$}
\myIf  {\parbox[t]{\dimexpr\linewidth-2.5em}{there is a  cycle $C_i$ that:
   \begin{enumerate} \item  has  length $2k_i$ and  contains $e^1_i, e^2_i$ as successive edges  in its positive traversal when we  start from $e_i^1$, \item   is  bichromatic  under the current coloring  \end{enumerate} }} \label{line:cyclecond}  
   \myThen {\parbox[t]{\dimexpr\linewidth-3.5em} {let $C_i$ be the unique such cycle \Comment{the uniqueness is due to the properness of the coloring}\\ 
   Successively assign to each edge of $\sco(e_i^1,C_i)$  a color selected u.a.r. from the set of   colors whose assignment  does not violate 4-acyclicity  \label{val:recolor}\\
   $i:=i+1$
   }}
   
   \myElse {\parbox[t]{\dimexpr\linewidth-3.5em}{$\text{success} := \text{\tt false}$; $i := s+1$  \strut}}   
   \myEndIf\EndWhile
\label{line:s-phase} \end{algorithmic}
\end{algorithm}
We call each iteration of the \textbf{while}-loop of \ValE\
  a {\em phase}. Also, the initial color assignment of line \ref{val:color} is called the {\em initial phase}.
 
Given an admissible sequence  $\mathcal{S}$, we say that an execution of \ValE$(\mathcal{S})$ terminates successfully if it goes through $\mathcal{S}$    without reporting $\text{success} := \text{\tt false}$, in which case it generates a uniquely defined sequence $\C = (C_1, \ldots, C_s)$ of cycles, which for $i=1, \ldots, s$ are bichromatic in the beginning of the corresponding phase, have length $2k_i\geq 6$ and contain $e_i^1, e_i^2$ successively in this order in their positive traversal starting from $e_i^1$. Such a sequence is a partial random variable defined for successful executions of \ValE$(\mathcal{S})$. For linguistic convenience, we call the sequence of cycles $\C$ a {\em success certificate} for the corresponding successful execution of \ValE$(\mathcal{S})$.

We first note that:

\begin{lemma}
\label{lem:randomness}
Let $e$ be an edge and $c$ be a color.  {\em At the time of a color assignment to $e$} the probability that the assigned color is $c$ is at most $$\frac{1}{\gamma(\Delta - 1)+1},$$ conditional on the colors of any assignments to any edges made before the said color assignment to $e$.
\end{lemma}

\begin{proof} The claim, even under the condition of previous color assignments, is a corollary of the fact that $\K$ colors are available and Lemma \ref{lem:sufficientcolors}, according to which $2(\Delta-1)$ colors should be avoided at every step to retain 4-acyclicity.  
\end{proof}

The lemma that follows was originally proved in Giotis et al. \cite[Lemma 5]{giotis2017acyclic}. The  proof  was also presented in Fialho et al. \cite[Appendix]{fialho2020new}. For completeness, we give it here as well.
\begin{lemma}\label{lem:boundadmissible}
Given an admissible  sequence ${\mathcal S}= (e^1_i, e^2_i, k_i), i=1, \ldots, s,$ we have that:    
\begin{multline}\label{eq:bound} \Pr\left[ \text{\ValE}  \text{\rm{ is successful on }} {\mathcal S}\right] \leq \\
 \left(\frac{1}{\gamma(\Delta-1)+1}\right)^s \prod_{i=1}^s\left(1 - \left(1 -\frac{1}{\gamma(\Delta-1)+1}\right)^{\Delta-1}\right)^{2k_i-3}. \nonumber\end{multline}		
\end{lemma}
\begin{proof} Given $\mathcal{S}= (e^1_i, e^2_i, k_i), i=1, \ldots, s$, we will compute the probability of existence of a success certificate for \ValE$(\mathcal{S})$.  Let such a certificate be $\C =(C_1, \ldots, C_s)$ and let $C_i$ be comprised of the successive edges $e_i^j, j =1, \ldots, 2k_i$ in its positive traversal starting from $e_i^1$. Note that $e_i^1, e_i^2$ are fixed and part of the input $\mathcal{S}$, whereas $e_i^j, j =3, \ldots, 2k_i$ are partial random variables that depend on the execution (they are uniquely defined for each successful execution). Let also $c_i^1, c_i^2$ be the colors of $e_i^1, e_i^2$ respectively (random variables as well).

Let $E_i^j, i=1, \ldots s, ,j=3, \ldots, 2k_i$ be the event that 
the color that $e_i^j$ has at  beginning the  of the $i$-th phase is $c_i^{(j+1) \bmod 2 +1}$ (the color of that one among the two $e_i^1, e_i^2$ which has the same parity as $e_i^j$). Then, 
by ordering the events by the time the corresponding assignments are made, in order to bound the conjunction 
$$\bigwedge_{\substack{i=1, \ldots, s\\ j=3, \ldots, 2k_i}	}E_i^j$$
it 
suffices to bound the probability of the product:
\begin{equation}\label{eq:hat}
\prod_{\substack{i=1, \ldots, s\\ j=3, \ldots, 2k_i}	}\Pr[E_i^j], \end{equation}
 The factors $ \Pr[E_i^j]$ in the product above are probabilities that are conditional only on color assignments that predate the assignments to $e_i^j$ of the colors they have at the beginning of the $i$-th phase. Therefore, we may consider $E_i^j$ as the event that the last color that $e_i^j$ gets before the start of the $i$-th phase is $c_i^{(j+1) \bmod 2 +1}$ (the color of that one among the two $e_i^1, e_i^2$ which has the same parity as $e_i^j$). If the edges $e_i^j$ were fixed, then by  Lemma \ref{lem:randomness}, we would have that $$  \Pr[E_i^j]\leq \frac{1}{\gamma(\Delta - 1)+1}. $$
  However, $e_i^j, i=1, \ldots, s, j=3, \ldots,  2k_i$ are random, so to bound  the probabilities   $\Pr[E_i^j]$,  i.e. to bound the probability of the existence of edges  $e_i^j$ with the right colors, we have to take into account all possibilities for the edges $e_i^j$. 
  
    Observe first that if $C =(v_1, e_1, v_2, e_2, v_3, e_3,  \ldots, v_k, e_k, v_1)$ 
  denotes an arbitrary  cycle with the vertices $v_i$ and the edges  $e_i$ 
  in the order of the cycle's  positive traversal starting from $e_1$, and if $v_1, e_1, v_2, e_2, v_3$ are fixed,  then $e_3$ has at most $\Delta-1$ 
  possibilities. In general  if  $$v_1, e_1, \ldots, v_{i-1},  e_{i-1}, v_i$$
  are fixed, then $e_i$ has at most $\Delta -1$ possibilities  if $2 < i < k$, whereas it    has just one possibility if $i=k$ 
  and $v_1, e_1, \ldots, e_{k-1},  v_k$ are fixed.  
  
  Having the above in mind,  we take into account  all possibilities for the edges $e_i^j$ successively for $i=1, \ldots, s$ and for each $i$, by traversing the edges $e_i^j, j=3, \ldots, 2k_i$ in the cycle's positive traversal and by fixing them, each in its turn. 
  Then   observe that  
  since at each step of the traversal of $C$, the possible edges that  may get the ``right" color are at most $\Delta -1$, the probability of at least one (and therefore, by properness, the  probability of exactly one)  getting this right color is at most \begin{equation*} 1- \left(1 -\frac{1}{\gamma(\Delta-1)+1}\right)^{\Delta-1}.\end{equation*}
  Indeed the probability that none gets the right color is at least $$ \left(1 -\frac{1}{\gamma(\Delta-1)+1}\right)^{\Delta-1}.$$  
  To get the latter lower bound,  instead of the probability of a conjunction, we considered the product of the probabilities of the conjuncts  because  we assumed that the events are in chronological order. Also observe that  the conditionals of $\Pr[E_i^j]$, which as we noted predate  
  the assignment to $e_i^j$ of the color it has at the beginning of phase $i$, also predate the color-assignment to each ``possibility" for $e_i^j$ we consider, because the ``possibilities" are candidates to be identified with $e_i^j$.

  The required follows because the factors $$\left(1 - \left(1 -\frac{1}{\gamma(\Delta-1)+1}\right)^{\Delta-1}\right)^{2k_i-3}, i=1, \ldots, s$$ cover all possibilities of the edges of all cycles but their last edge, whereas the single possibility for their last edge is covered by the factor $\left(\frac{1}{\gamma(\Delta-1)+1}\right)^s.$ 
 \end{proof}

  \begin{corollary}\label{cor:boundadmissible}
Given an admissible  sequence ${\mathcal S}= (e^1_i, e^2_i, k_i), i=1, \ldots, s,$ we have:    
\begin{equation*}\label{eq:bound} \Pr\left[ \text{\ValE}  \text{\rm{ is successful on }} {\mathcal S}\right] 
\leq   \left(\frac{1}{\Delta-1}\right)^s  \prod_{i=1}^s \left(\frac{1}{\gamma}\left(1 -  e^{-\frac{1}{\gamma}} \right)^{2k_i-3}\right).
\end{equation*}
\end{corollary}

\begin{proof} 
Use the  the inequality $$1-x>e^{-\frac{x}{1-x}}, \forall x<1, x \neq 0,$$ 
for $x= \frac{1}{\gamma(\Delta -1)+1}$
(see  also \cite[Corollary 2]{giotis2017acyclic}).
 \end{proof}
 
 In the sequel, a marked forest is a forest whose internal edges are marked with pairs $(e, C)$ where $C$ is a cycle that contain the edge $e$. We assume  that a marked forest has all the properties of a witness forest, i.e. that the cycle-marks have even length $\geq 6$, that the outdegrees of all internal nodes are  equal to the length of their cycle-mark minus 2, that the roots are internal nodes and that the items \ref{ite:fir} -- \ref{ite:thi} of Corollary~\ref{cor:witnessProp} are satisfied, lest the probability of the marked forest being a witness forest is zero.
 
Given a marked forest $\mathcal{F}$ whose  sequence of marks (in the depth-first ordering of its internal nodes)  is $(e_i, C_i), i=1, \ldots,s$ {\em the corresponding admissible sequence} $\mathcal{S}_\mathcal{F}$ is obtained by letting  $e_i^1$ be  $e_i$, $e_i^2$ be the edge of $C_i$  following $e_i$ in $C_i$'s positive traversal and $k_i$ be the half-length of $C_i$, $i=1, \ldots, s$. Given an admissible sequence $\mathcal{S}$ let $\frak{F}_\mathcal{S}$ be the {\em class} of all marked  forests $\mathcal{F}$ as above  such that $\mathcal{S}_\mathcal{F} = \mathcal{S}.$

Giotis et al.  prove  the following fact (see\cite[Proposition 1]{giotis2017acyclic}); for completeness the proof is also included here.

\begin{lemma}\label{lem:proposition}  
Given an admissible sequence ${\mathcal S},$ we have:

\begin{multline}\label{eq:disjoint1}
\sum_{{\mathcal F} \in {\frak F}_{\mathcal S}} \Pr[ \text{ \EC\ executes with  witness forest } {\mathcal F}] \leq \\
	\Pr[\text{\ValE\ is successful on input } {\mathcal S}].\end{multline}
\end{lemma}

\begin{proof} We prove first that the probability  $P$   for at least one ${\mathcal F} \in {\frak F}_{\mathcal S}$ being a witness forest of \EC\ is bounded from above by the probability that \ValE\ is successful on input  ${\mathcal S}$. 
For this it is sufficient to notice that if the random choices made by an execution of \EC\ that produces an arbitrary  ${\mathcal F} \in {\frak F}_{\mathcal S}$ 
are made by \ValE\ on input ${\mathcal S}_{\mathcal F} = {\mathcal S}$, 
then \ValE\ is successful. The result now follows by observing the  events that 
${\mathcal F}$ is a witness forest of \EC\ for various 
${\mathcal F} \in {\frak F}_{\mathcal S}$ are mutually exclusive, therefore $P$ can be written as the sum in the lhs of the Inequality \eqref{eq:disjoint1}.\end{proof}

So by Lemma \ref{lem:proposition} and  Corollary \ref{cor:boundadmissible} we have:

\begin{corollary}\label{cor:boundadmissible2}
Given an admissible  sequence ${\mathcal S}= (e^1_i, e^2_i, k_i), i=1, \ldots, s,$ we have:

\begin{multline}\label{eq:disjoint2}
\sum_{{\mathcal F} \in {\frak F}_{\mathcal S}} \Pr[ \text{ \EC\ executes with  witness forest } {\mathcal F}] \leq \\
\left(\frac{1}{\Delta-1}\right)^s  \prod_{i=1}^s \left(\frac{1}{\gamma}\left(1 -  e^{-\frac{1}{\gamma}} \right)^{2k_i-3}\right).\end{multline}\end{corollary}

In the sequel,  if $\F$ is a forest $|\F|$ denotes its number of nodes (leaves included), which is often also  denoted by $n$, and $|\F|_{\mathrm{in}}$ denotes the number of its internal nodes. 
If $\C$ is a sequence of cycles, $|\C|$ denotes  the length of this sequence,  and if $C$ is a cycle $|C|$ denotes the number of edges in $C$.

\begin{lemma}\label{lem:finalprob}
Given an {\rm unmarked}  forest $\F$  with   outdegrees  of its  internal nodes  $2k_i, i=1,\ldots,|\F|_{\mathrm{in}}$, let $V_{\F}$ be the event that $\F$ can be marked in a way that it is a witness forest, then 
\begin{equation} \label{eq:VF} \Pr[V_{\F}] = {\rm O} \Bigg( 
\prod_{i=1}^{|\F|_{\rm in}}\frac{1}{\gamma}
\left(1-e^{-1/\gamma}\right)^{2k_i -3} 
\Bigg) 
= {\rm O} \Bigg(\left(1-e^{-1/\gamma}\right)^{(3/4)|\F|}\Bigg). \end{equation}

\end{lemma}
\begin{proof} Fix the given unmarked forest. We first assume that it is a tree.

We start by proving the first equality. 
  The length of the cycle-mark of any internal node diminished by two is equal to the outdegree of the node, which is fixed.

The number of possible edge-marks of its root is a constant, since the number of edges is considered a constant.  Once the edge-mark of the root is fixed, the number of possible next edges in the cycle-marks' positive  traversal is at most $\Delta-1$, because the next edge in the positive traversal may start from exactly one of the endpoints of the edge-mark. 

Now given  the cycle-mark of an internal node, the edge-marks of its children become fixed, and again  for each such edge-mark of a child, there are at most $\Delta-1$ possibilities for the next edge in the positive traversal of the children's  cycle-mark.

Thus altogether there are at most $(\Delta-1)^{|\F|_{\rm in}}$ possible admissible sequences $\mathcal{S}$ so that the marked forest $\F$ obtained after marking the given unmarked forest satisfies $\F = \F_{\mathcal {S}}.$ The first equality now follows from Corollary \ref{cor:boundadmissible2}.

As for the second equality, distribute the numbers $\left(1-e^{-1/\gamma}\right)^{2k_i -3 }$, associated with each 
internal node of $\F$,  equally to each of its $2k_i-2$  children, to get that all nodes $u$ of $\F$ except the root are now associated with a   number at most 
$\left(1-e^{-1/\gamma}\right)^{\frac{2k_i -3}{2k_i-2} }$ (recall that $\gamma \geq 1$). But $\frac{2k_i -3}{2k_i-2} = 1 -\frac{1}{2k_i-2}$ is at least $3/4$, because  $k_i$ is at least 3.

When $\F$ is not a tree but a forest, work with each tree separately.
\end{proof}


\begin{lemma}\label{lem:boundtrees}
If $t_n$ is the number of  {\rm unordered,  unmarked forests} with $n$ nodes and with a constant number of trees, whose internal nodes have an even outdegree $\geq 4$, then $\limsup_n t_n^{1/n} = \rho ^{-1}$,  for some $\rho \geq 0.6677$.
	\end{lemma}
	
	The proof of the above Lemma is computational and makes use of a conclusion of the ``smooth implicit-function schema" given by Flajolet and Sedgewick in \cite[Theorem VII.3]{Flajolet:2009:AC:1506267}. In order not to interrupt the flow of the proof of our \ref{maintheorem}, we give its proof in an Appendix.

	We now give the proof of Lemma \ref{factt1}, and thus the proof of  our \ref{maintheorem} is concluded.
	
	\begin{proof}[Proof of Lemma \ref{factt1}] It is enough to show that
	  $$\sum_{\substack{{|\F|=n}\\{\F \mbox{ is not marked}}}}\Pr[V_{\F}]$$	
	  is inverse exponetial in $n$.	Because of Lemma \ref{lem:finalprob}, we have:
	
	$$\sum_{\substack{{|\F|=n}\\{\F \mbox{ is not marked}}}}\Pr[V_{\F}] = t_n {\rm O}\Bigg(  \left(1-e^{-1/\gamma}\right)^{(3/4)n} \Bigg)$$
Now just observe that if $\gamma \geq 1.142$	then $\left(1-e^{-1/\gamma}\right)^{3/4} < 0.6677$. Then use Lemma~\ref{lem:boundtrees}.
	\end{proof}

\section*{Acknowledgments}

We are indebted to the anonymous referees who pointed out mistakes in previous versions of this work. 

\section*{Data Availability Statement}
The Maple file for the computations in the Appendix, together with their specifications, are provided  \href{https://drive.google.com/file/d/1ZRTi95dtOXLDpn2KKEApK2Q1sPOTVK5k/view?usp=sharing}{here}.


\begin{thebibliography}{10}

\bibitem{alon1991acyclic}
Noga Alon, Colin McDiarmid, and Bruce Reed.
\newblock Acyclic coloring of graphs.
\newblock {\em Random Structures \& Algorithms}, 2(3):277--288, 1991.

\bibitem{alon2001acyclic}
Noga Alon, Benny Sudakov, and Ayal Zaks.
\newblock Acyclic edge colorings of graphs.
\newblock {\em Journal of Graph Theory}, 37(3):157--167, 2001.

\bibitem{DBLP:journals/ejc/EsperetP13}
Louis Esperet and Aline Parreau.
\newblock Acyclic edge-coloring using entropy compression.
\newblock {\em Eur. J. Comb.}, 34(6):1019--1027, 2013.

\bibitem{fialho2020new}
Paula~MS Fialho, Bernardo~NB de~Lima, and Aldo Procacci.
\newblock A new bound on the acyclic edge chromatic number.
\newblock {\em Discrete Mathematics}, 343(11):112037, 2020.

\bibitem{fiam}
J.~Fiam\v{c}ik.
\newblock The acyclic chromatic class of a graph (in {R}ussian).
\newblock {\em Math. Slovaca}, 28:139--145, 1978.

\bibitem{Flajolet:2009:AC:1506267}
Philippe Flajolet and Robert Sedgewick.
\newblock {\em Analytic Combinatorics}.
\newblock Cambridge University Press, New York, NY, USA, 1 edition, 2009.

\bibitem{giotis2017acyclic}
Ioannis Giotis, Lefteris Kirousis, Kostas~I Psaromiligkos, and Dimitrios~M
  Thilikos.
\newblock Acyclic edge coloring through the {L}ov{\'a}sz local lemma.
\newblock {\em Theoretical Computer Science}, 665:40--50, 2017.
\newblock Correction in: \href{http://arxiv.org/abs/1407.5374}{\tt
  arXiv:1407.5374}.

\bibitem{Gutowski2018}
Grzegorz Gutowski, Jakub Kozik, and Xuding Zhu.
\newblock Acyclic edge coloring using entropy compression.
\newblock In {\em Abstracts of the 7th Polish Combinatorial Conference}, 2018.
\newblock Available:
  \href{https://7pcc.tcs.uj.edu.pl/program.php}{\url{https://7pcc.tcs.uj.edu.pl/program.php}}.

\bibitem{harvey2015algorithmic}
Nicholas~JA Harvey and Jan Vondr{\'a}k.
\newblock An algorithmic proof of the lov{\'a}sz local lemma via resampling
  oracles.
\newblock In {\em 2015 IEEE 56th Annual Symposium on Foundations of Computer
  Science}, pages 1327--1346. IEEE, 2015.

\bibitem{molloy1998further}
Michael Molloy and Bruce Reed.
\newblock Further algorithmic aspects of the local lemma.
\newblock In {\em Proceedings of the thirtieth annual ACM symposium on Theory
  of computing}, pages 524--529, 1998.

\bibitem{moser2009constructive}
Robin~A Moser.
\newblock A constructive proof of the {L}ov{\'a}sz {L}ocal {L}emma.
\newblock In {\em Proceedings of the 41st annual ACM Symposium on Theory of
  Computing}, pages 343--350. ACM, 2009.

\bibitem{moser2010constructive}
Robin~A Moser and G{\'a}bor Tardos.
\newblock A constructive proof of the general {L}ov{\'a}sz {L}ocal {L}emma.
\newblock {\em Journal of the ACM (JACM)}, 57(2):11, 2010.

\bibitem{ndreca2012improved}
Sokol Ndreca, Aldo Procacci, and Benedetto Scoppola.
\newblock Improved bounds on coloring of graphs.
\newblock {\em European Journal of Combinatorics}, 33(4):592--609, 2012.

\bibitem{vizing1965critical}
Vadim~G Vizing.
\newblock Critical graphs with a given chromatic class.
\newblock {\em Diskret. Analiz}, 5(1):9--17, 1965.

\end{thebibliography}

\section*{Statements and Declarations}
\noindent{\bf Funding.} The authors declare that no funds, grants, or other support were received during the preparation of this manuscript.

\noindent{\bf Competing Interests.} The authors have no relevant financial or non-financial interests to disclose.

\section*{Appendix: Asymptotic estimate for a species of trees}
In this Appendix, where we follow the notation and terminology of Flajolet and Sedgewick in \cite{Flajolet:2009:AC:1506267}, we will compute upper and lower bounds of the radius of convergence of the counting generating function (GF)  of unordered trees with $n$ nodes whose outdegrees are even and $\geq 4$. We denote this class by $\T$ and its GF by $\nicetrees(z)$.  Note that since the forests we consider have a finite number of trees and since we are interested in exponential approximations, working with trees rather than forests suffices. Let us also add that only the lower bound will be needed for the proof of  Lemma \ref{lem:boundtrees}, however the upper bound is necessary to limit the search space.

We begin by giving a functional equation for $T(z)$, which will serve as the basis for our analysis.

\begin{lemma}\label{lemma:spec1}
The generating function $\nicetrees(z)$ obeys the functional equation:
\begin{equation}\label{eq:spec}
 T(z) = \frac{z}{2} \exp\left(\sum\limits_{j \geq 1} \frac{\nicetrees(z^j)}{j}\right) + \frac{z}{2}\exp\left(\sum\limits_{j \geq 1} (-1)^j \frac{\nicetrees(z^j)}{j} \right)  
-\frac{z}{2}\left({T^2(z)} + {\nicetrees(z^2)} \right).
\end{equation}
\end{lemma}
\begin{proof}
It is a standard result  \cite[Figure I.13, p. 65]{Flajolet:2009:AC:1506267})
that the class $\mathcal{H}$ of (rooted)  {\em unordered} trees is 
the product of a singleton and the class of multisets of elements in $\mathcal{H}$:
\begin{equation}\label{eq:specAllTrees}
    \mathcal{H} = \mathcal{Z} \times \MSet(\mathcal{H}).
\end{equation}
Our specification is a straightforward modification of the above.
For a given combinatorial class $\mathcal{X}$, denote 
by $\MSetOdd(\mathcal{X})$ the class of multisets of elements drawn from $\mathcal{X}$
in which multisets of odd size are weighted by a factor of $-1$.
We can use this construction ``delete'' from the recursive 
specification in Equation \eqref{eq:specAllTrees} those
trees whose roots are of odd degree, after 
appropriately weighting the whole construction by $1/2$ to avoid overcounting. 
In this way we obtain the class $\mathcal{H}_{even}$ of {\em  unordered}
trees whose internal nodes all have even degree:

\begin{equation*} 
\mathcal{H}_{even} = \mathcal{Z} \times \frac{\MSet(\mathcal{H}_{even}) + 
\MSetOdd(\mathcal{H}_{even})}{2}.
\end{equation*}
Further deletion from the recursive specification of $\mathcal{H}_{even}$ of
trees with roots of  degree $2$ yields:
\begin{equation*}
\nicetreesc = \mathcal{Z} \times \frac{\MSet(\nicetreesc) + 
\MSetOdd(\nicetreesc)}{2} - 
\mathcal{Z}\times\MSetTwo\left(\nicetreesc\right), 
\end{equation*}
where $\MSet_2(\mathcal{X})$ is the class of multisets drawn from $\mathcal{X}$ with exactly two elements.

Using the standard ``dictionary'' in \cite[Figure I.18, p. 93]{Flajolet:2009:AC:1506267}, we then translate the above construction to a functional equation for
the GF  $\nicetrees(z)$ of the class $\nicetreesc$ ---the GF's  for $\MSet(\mathcal{X}),$ $\MSetOdd(\mathcal{X})$ and $\MSet_2(\mathcal{X})$ for a class $\mathcal{X}$ with GF $X(z)$ are $$
\exp\left(\sum\limits_{j \geq 1} \frac{X(z^j)}{j}\right), \exp\left(\sum\limits_{j \geq 1} (-1)^j \frac{X(z^j)}{j}\right) \mbox{ and } \frac{1}{2}\left(X^2(z) + X(z^2) \right),$$ respectively. 
\end{proof}

We continue by considering two classes of {\em ordered} trees.  The first class 
is the class of ordered full binary trees, i.e. binary trees whose internal nodes have outdegree exactly 2. The second one 
is the class of ordered trees  whose internal nodes have outdegrees an even integer $\geq 4$.  Let $B(z)$ and $C(z)$ be their respective GF's. The first  class  is introduced because we can easily get a closed-form formula for $B(z)$ with coefficients that can be expressed in terms of Catalan numbers. The second  class   is introduced only to be used as an intermediary in Lemma \ref{lem:planeTreeCoeffIneq}  to prove that $B(z)$ dominates coefficient-wise $T(z)$.

We can easily see that $B(z)$ and $C(z)$ satisfy the functional equations:
\begin{equation}\label{eq:defB}
    B = z(1+B^2).
\end{equation}
and 
\begin{equation}\label{eq:defC} C(z) = z\left( 1+ C^4(z) + C^6(z) + \cdots \right) = z\left( 1+ \frac{C^4(z)}{1 - C^2(z)}\right),\end{equation}
and observe that from the functional equation \eqref{eq:defB} we get that 
\begin{equation}\label{eq:B2}
B(z) = \frac{1-\sqrt{1-4z^2}}{2z}
\end{equation}
(the solution of the  quadratic functional equation with the + sign is not retained since it gives rise to a $B(z)$ not analytic at 0).
Therefore  $B(z)$ has  radius 1/2 (also $C(z)$ can be proved to have a radius of exactly 1/2, but we do not need this fact). Finally, the coefficient of $B(z)$ for an odd $2n+1, n\geq 0,$ is the Catalan number $C_n$.

Now we prove:

\begin{lemma}\label{lem:planeTreeCoeffIneq} $B(z)$ dominates coefficient-wise C(z), i.e.,
 $[z^n] C(z) \leq [z^n] B(z)$ and $C(z)$ dominates coefficient-wise $T(z).$
\end{lemma}
\begin{proof} The assertion comparing $C(z)$ with $T(z)$ is obvious, since the only difference between the respective classes of trees is that the $C(z)$ refers to ordered trees and $T(z)$ to unordered. 

Throughout the rest of the proof, we will  view
$C:=C(z)$ and $B:=B(z)$ as elements of the ring
$\mathbb{C}[[z]]$ of complex  formal power series in $z$.

Eliminating $z$ from both \cref{eq:defC,eq:defB} leads to
\begin{equation*}
    (C^2 + CB - 1)(C^2B + C - B) = 0.
\end{equation*}
The first factor $(C^2 + CB - 1)$ has constant term $-1$ and is a
unit in $\mathbb{C}[[z]]$. For the product to be zero,
the second factor must therefore be zero, so
\begin{equation*}
    B = \frac{C}{1-C^2}
\end{equation*}
and 
\begin{equation*}
    B - C = \frac{C^3}{1-C^2} = C^3 + C^5 + C^7 + \dots .
\end{equation*}
The right-hand side of the last equality is finally seen to have non-negative coefficients,
which implies the desired inequality. 
\end{proof}

The following lemma shows that $\nicetrees(z)$ is analytic 
in some non-trivial disc in $\mathbb{C}$ and provides
some rough  bounds on its radius of convergence. 

\begin{lemma}\label{lemma:tIsAnalyt}
The generating function $\nicetrees(z)$ is analytic
in some disk $D \subseteq \mathbb{C}$ centered at $0$ and 
of radius $\rho$ such that $1/2 \leq \rho < 1$
\end{lemma}
\begin{proof}
Since by Lemma   \ref{lem:planeTreeCoeffIneq}, $T(z)$ is dominated coefficient-wise by $B(z)$ and the latter has radius exactly 1/2 (see its closed form in Equation \eqref{eq:B2}), the radius of convergence of the former is at least $1/2$.

$\nicetrees(z)$ also has a
radius of convergence $< 1$, because for odd $n\geq 5$, the
coefficients $[z^n]\nicetrees(z)$ are at least  1 (exactly 1, for $n=5$)  and strictly increasing.

\end{proof}

Before proceeding to the main body of this section, first note that if we formally unroll the exponentials in the rhs of Equation \eqref{eq:spec} as power series, all negative signs inside the exponentials cancel out. Also, the negative signs outside the exponential are absorbed by positive ones. 

Finally, if we further unroll all $T(z^j), j\geq 1,$ as power series in both sides of this equation,  then because of the factor $z$ at the rhs
we get a recurrence for $[z^n] T(z), n \geq 1 $, i.e. each $[z^n] T(z), n \geq 1,$ can be expressed in terms of lower-order coefficients of $T(z)$. This suffices to
fix $\nicetrees(z)$ as the sole solution of the functional equation  with $T(0)=0$. Furthermore, it
allows us to compute truncations of any desired order for $T(z)$, and therefore for  $T(z^j), j\geq 2,$ as well. 
Actually, it can be easily seen that the recurrences are polynomial expressions with nonnegative coefficients. These facts also hold for all implicitly defined functions we examine in this work.

The following two lemmas will be useful in the sequel.

The first one allows us to relate $f(x^i)$ to $f(x^2)$ of a power series $f(z)$
for nonnegative reals $x$ inside an appropriate interval. We omit the easy proof.

\begin{lemma}\label{lemma:powerbounds}
Let $f(z)$ be a power series with nonnegative coefficients and
$[z^0] f(z) = 0$, $[z^1] f(z) = 1$ and radius of convergence $0 < r < 1$. 
Then $f(x^i) \leq f(x^2)^{i/2}$ for $i \geq 2, x \in [0,r^{1/2})$.
\end{lemma}

The second one gives us  the radius of convergence for
implicitly-defined generating functions. It is a consequence  of the smooth implicit-function schema presented by Flajolet and Sedgewick in \cite[Theorem VII.3, p. 468]{Flajolet:2009:AC:1506267}. 
For reasons of completeness, we include its proof. 

\begin{lemma}\label{lemma:implicitsing}
Let $G(z,w)$ be an entire  function such that:
\begin{itemize}
    \item $[z^nw^m]G(z,w) \geq 0$ and nonzero for some $m\geq 1$,
    \item $G(0,w) = 0$, for all $w$,
 \end{itemize}
  Then there exists an analytic solution $y(z) = \sum_{n=1}^{\infty}a_nz^n$ to the functional
equation $y=G(z,y)$ with $y(0)=0$, with uniquely determined $a_n \geq 0$,   and  positive finite radius of convergence $r$, such that   $\lim_{x \rightarrow r^{-}} y(x) = s$ exists and is finite.  Moreover $r,s$ satisfy the {\em characteristic} system,
\begin{equation*}
    G(r,s) = s,  \ 
    G_w(r,s) = 1.
\end{equation*}
\end{lemma}
\begin{proof}
 Because $G(0,w) = 0$, for all $w$, we conclude that $z$ can be factored out of $G$. Therefore $G_w(0,0) = 0$. By the implicit function theorem, since $G_w(0,0) = 0 \neq 1$, there exists
a unique solution $y= y(z)$  to the  equation $y-G(z,y)=0$ which is analytic around $0$. Since $z$ can be factored out of $G$, we have that $[z]y(z)$ is given by $[zw^0]G(z,w)$;  further, since $[z^nw^m]G(z,w)$ is positive for some $m\geq 1$ we conclude that each 
coefficient $[z^j]y(z), j\geq 2,$ can be expressed by a recurrence in terms of lower-order ones, therefore they are uniquely determined. Since the coefficients of $G$ are nonnegative, the recurrences also show the $[z^j]y(z)$ to be nonnegative.  Now, obsereve that the mapping $x \mapsto G_w(x,y(x))$ is  strictly increasing for
nonnegative real values $x$, as long as $y(z)$ is regular at $z=x$. 
If $G_w(z,y(z)) < 1$, the implicit function theorem guarantees that
$y(z)$ is regular in a neighbourhood of $z$. Since $G_w(0,y(0)) = 0$ and $x \mapsto G_w(x,y(x))$ is strictly increasing and goes to infinity with $x\rightarrow \infty$, there
exists some finite limit point $r$ such that 
$\lim_{x\rightarrow r^-} y(x) = s$ is finite 
and satisfies $G_w(r,s)=1$. If $y(z)$ was regular at $z=r$,
then we would have:
\begin{align}
y_z(r) = G_z(r,y(r)) + G_w(r,y(r))y_z(r) = G_z(r,y(r)) + y_z(r)
\end{align}
which would imply $G_z(r,s) = 0$ contradicting
that $r,s$ are positive  and  $[z^nw^m]G(z,w)$ are not all zero.\end{proof}


We will now place  the radius of
convergence for $\nicetrees(z)$ into a narrow interval.

\begin{theorem}\label{thm:boundstoradius}
The radius convergence $\rho$ of  $T(z)$ satisfies: 
\begin{equation*}
    0.6677 \leq \rho \leq 0.6678.
\end{equation*}
\end{theorem}
\begin{proof}

At first, consider the two variable function:

\begin{align}\label{eq:G}
\nonumber G(z,w) =   &\frac{z}{2} 
\ \exp\left( w + \sum\limits_{j \geq 2} \frac{\nicetrees(z^j)}{j}\right) 
+ \frac{z}{2}\exp\left(-w + \sum\limits_{j \geq 2} (-1)^j \frac{\nicetrees(z^j)}{j}\right) \\ &
- \frac{z}{2}\left({w^2}
+ {\nicetrees(z^2)} \right),
\end{align}
and we observe that by Lemma \ref{lemma:spec1}, $T(z) = G(z,T(z))$.

Let $N$ be a fixed (large) even integer.

We begin with computing the upper bound.   Recalling that we can compute coefficients of $T(z)$ up to any order, we compute a truncation of $T(z)$ to powers at most  $N$, which we call $L(z)$ ($L(z)$ is a polynomial of degree $N-1$, since even powers of $z$ in $T(z)$ vanish ---there exist no rooted trees with even outdegrees and even number of nodes). Then in $G$ above, we plug  $L(z^j)$ for $T(z^j)$  and restrict the sums in the expoentials to $j=2, \ldots, N$. We thus obtain some  $\UG(z,w)$, where:
\begin{align}\label{eq:UG}
\nonumber \UG(z,w) =   &\frac{z}{2} 
\ \exp\left( w + \sum\limits_{j =2}^{N} \frac{L(z^j)}{j}\right) 
+ \frac{z}{2}\exp\left(-w + \sum\limits_{j =2}^{N} (-1)^j \frac{L(z^j)}{j}\right) \\ &
- \frac{z}{2}\left({w^2}
+ {L(z^2)} \right).
\end{align}
We claim that $\UG(z,w)$ satisfies  all conditions for Lemma \ref{lemma:implicitsing}.  Indeed, first  observe that 
$\UG(z,w)$  is  entire,  because of the truncations. The first bullet in Lemma~\ref{lemma:implicitsing} follows by expanding the two exponentials in $\UG(z,w)$ as power series and observing, as before,  that terms with negative signs either cancel out or are absorbed  by ones with positive signs. The second  bullet  is obvious.  So, let $r,s$ be the positive real numbers that satisfy 
$$\UG(r,s)=r,\UG_w(r,s)=1.$$
 Observe now that  the solution $\hat{y}(z)$ to $\hat{y}(z) = \UG(z,\hat{y}(z))$
satisfies $[z^n] \hat{y}(z) \leq [z^n] \nicetrees(z)$, because  $\UG$ was obtained by truncations in $G$ and both  $[z^n] \hat{y}(z)$ and  $[z^n] \nicetrees(z), n\geq 1,$ are  obtained by polynomial recurrences with nonnegative coefficients defined via $\hat{y} = \UG(z,\hat{y}(z))$ and $T(z) = G(z, T(z))$, respectively. So $r$ is an upper bound on $\rho$. Finally, we choose, for example,  $N=100$, and  use a computer algebra system (in our case, Maple) to solve the system and find $r \leq \hat{\rho} := 0.6678$.

For the lower bound, first recall that
by Lemma \ref{lem:planeTreeCoeffIneq}     $B(z)$, the GF of full ordered binary trees, dominates coefficient-wise $T(z)$. Therefore the unique solution analytic at $0$ to the functional equation $y(z) = H(z, y(z))$, where  $H(z,w)$ is defined by: 
\begin{equation}\label{eq:upperBoundingSystem}
\begin{split}
    H(z,w) &=  \frac{z}{2} 
        \exp\left( w + \sum\limits_{j = 2}^{N} \frac{\nicetrees(z^j)}{j} 
        + \sum\limits_{j > N} \frac{B(z^j)}{j} \right)  \\ 
      &+\frac{z}{2}\exp\left(-w + \sum\limits_{j = 2}^{N} (-1)^j \frac{\nicetrees(z^j)}{j} 
      + \sum\limits_{j > N} \frac{(-1)^j B(z^j)}{j} \right) \\&
       -\frac{z}{2}\left(w^2 + T(z^2) \right)
\end{split}
\end{equation}
dominates coefficient-wise $\nicetrees(z)$ and therefore its radius of convergence is at most $\rho$. Indeed, just recall that unrolling the exponentials, the negative signs are cancelled.  Also observe that 
\begin{equation}\label{eq:upperBoundingSystem2}
\begin{split}
    H(z,w) &=  \frac{z}{2} 
        \exp\left( w + \sum\limits_{j = 2}^{N} \frac{L(z^j) +(T(z^j) - L(z^j))}{j}
        + \sum\limits_{j > N} \frac{B(z^j)}{j} \right)  \\ 
      &+\frac{z}{2}\exp\left(-w + \sum\limits_{j = 2}^{N} (-1)^j \frac{L(z^j) +(T(z^j) - L(z^j))}{j} 
      + \sum\limits_{j > N} \frac{(-1)^j B(z^j)}{j} \right) \\
      &-\frac{z}{2}\left(w^2 + L(z^2) +(T(z^2) - L(z^2))\right)
\end{split}
\end{equation}
 
Now, besides $L(z)$, the truncation of $T(z)$,  we also consider $M(z)$, the truncation of $B(z)$ up to powers at most $N$. Recall that $[z^{2n+1}]B(z)$ is the $n$-th Catalan number. 
We then  bound from above  
$\nicetrees(z^j) - L(z^j), j\geq 2,$  as well as $B(z^j),j \geq 2,$
 on the  interval $[0,\rho)$,  where all  $\nicetrees(z^j), j\geq 2,$ and $B(z^2)$ are analytic (because $\rho \leq  \hat{\rho} =  0.6678 < \sqrt(1/2)$), as follows.  

Observe that for $x \in [0,\rho)$ and $j\geq 2$, and because $\rho <1$ and  $B(z)$ bounds coefficient-wise $T(z)$, we have:

\begin{equation}\label{eq:error1}
	\nicetrees(x^j) - L(x^j) \leq B(x^j) - M(x^j)  \leq B(x^2) - M(x^2) \leq B(\hat{\rho}^2) -M(\hat{\rho}^2) :=e.
	\end{equation}

 By Lemma \ref{lemma:powerbounds},  we also have:

\begin{equation}\label{eq:error2}
     \frac{B(x^j)}{j} \leq  \frac{b^{j/2}}{j}
   \end{equation}
for $x \in [0,\hat{\rho})$, where $b := B(\hat{\rho}^2)$.

We plug the errors given in  Equations \eqref{eq:error1} and \eqref{eq:error2} into the rhs of Equation \eqref{eq:upperBoundingSystem2}, to obtain the function:
\begin{equation}\label{eq:Gcheck}
\begin{split}
    \LG(z,w) &=  \frac{z}{2} 
        \exp\left( w + \sum\limits_{j = 2}^{N} \frac{L(z^j) +e}{j}
        + \sum\limits_{j > N} \frac{b^{j/2}}{j} \right)  \\ 
      &+\frac{z}{2}\exp\left(-w + \sum\limits_{j = 2}^{N} (-1)^j \frac{L(z^j) +e}{j} 
      + \sum\limits_{j > N} \frac{(-1)^j b^{j/2}}{j} \right) \\
      &-\frac{z}{2}\left(w^2 + L(z^2) +e\right).\\  
    \end{split}
 \end{equation} 
 Analogously to $\hat{G}$ we examined for the upper bound, we can show that the  entire $\check{G}$  satisfies  the conditions of Lemma \ref{lemma:implicitsing}. So let $\check{y}(z)$ be the analytic around 0 function implicitly defined by  $y(z) = \hat{G}(z, y(z)$ and also let $r,s = \lim_{x\rightarrow r^{-}}\check{y}(x)$ be the unique solution of the characteristic system $$\check{G}(r,s)=s, \ \check{G}_w(r,s) =1$$.

 Now, if $h_{nm}$ and $\check{g}_{nm}$  are the Taylor coefficients of the two-variable functions 
  $H$ and   $\check{G}$, respectively, 
 we have that $\check{g}_{nm} \geq   h_{nm} $. Indeed, the appearance of  the error $e$ in $-\frac{z}{2}\left(w^2 + L(z^2) +e\right)$ vanishes when we differentiate. Therefore we may assume that the error $e$ appears in the exponentials only. So does $b$,  so $e$ and $b$  contribute multiplicatively to the coeffcients of $\check{G}$, so  $$\frac{\partial^n \partial^m \check{G}(0,0)}{\partial z^n \partial w^m} \geq \frac{\partial^n \partial^m H(0,0)}{\partial z^n \partial w^m}, \forall n, m.   $$
   
Therefore the coefficients $[z^j]\check{g}(z)$, which are given by a recurrence that entails $\check{g}_{nm}$, 
 are at least as large as the corresponding coefficients of the solution of $w = H(z,w)$, which are given by a recurrence that entails $h_{nm}$, 
 and therefore at least as large as the coefficients of $T(z)$, so $r \leq \rho$.

   We finally take, for example, $N=100$ and use Maple to find $r \geq \check{\rho} := 0.6677.$ \end{proof}

The Cauchy-Hadamard theorem that for any series $\sum_{n=0}^{\infty}a_nz^n$ its radius of convergence $r$  satisfies $ r^{-1} = \limsup_n |a_n|^{1/n}$ then proves  Lemma \ref{lem:boundtrees}, and threfore  our \ref{maintheorem}.

The Maple computations, and their specifications, are provided \href{https://drive.google.com/file/d/1ZRTi95dtOXLDpn2KKEApK2Q1sPOTVK5k/view?usp=sharing}{here}.

\end{document}